\documentclass[11pt,reqno]{amsart}

\usepackage{lmodern}
\usepackage[T1]{fontenc}

\usepackage[margin=1in]{geometry}
\usepackage{paralist}

\usepackage[colorlinks,pagebackref=true]{hyperref}

\newcommand{\arxiv}[1]{\href{http://arxiv.org/abs/#1}{{\tt arXiv:#1}}}

\usepackage{mathrsfs}
\usepackage{multicol}
\usepackage{fancyvrb}
\usepackage{color}
\usepackage{amssymb,latexsym,amsbsy,amsthm,amsxtra,amscd,amsfonts,amsthm,amsmath,verbatim}

\newcommand{\ncom}{\newcommand}
\ncom{\bq}{\begin{equation}}
\ncom{\eq}{\end{equation}}
\ncom{\beqn}{\begin{eqnarray*}}
\ncom{\eeqn}{\end{eqnarray*}}
\ncom{\beq}{\begin{eqnarray}}
\ncom{\eeq}{\end{eqnarray}}
\ncom{\been}{\begin{enumerate}}
\ncom{\eeen}{\end{enumerate}}
\ncom{\olin}{\overline}
\ncom{\f}{\frac}
\ncom{\rar}{\rightarrow}

\def\nno{\nonumber}

\makeatletter

\@addtoreset{equation}{section}
\def\theequation{\thesection.\@arabic \c@equation}

\def\@citecolor{blue}
\def\@linkcolor{blue}
\def\@urlcolor{blue}

\makeatother

\usepackage{amssymb,latexsym,color,
amsxtra,amscd,amsfonts,amsthm,amsmath,verbatim,epsfig}

\usepackage{multicol}
\usepackage{fancyvrb}
\usepackage[colorlinks,pagebackref=true]{hyperref}
\def\@citecolor{blue}
\def\@linkcolor{blue}
\def\@urlcolor{blue}

\def\theequation{\arabic{equation}}
\def\theequation{\thesection.\arabic{equation}}
\numberwithin{equation}{section}

\def\charac{\operatorname{char}}

\def\dim{\operatorname{dim}}

\def\ker{\operatorname{ker}}

\def\supp{\operatorname{Supp}}

\newcommand{\kk}{\Bbbk}

\def\lrar{{\longrightarrow}}
\def\Z{\mathbb Z}

\def\A{\mathbb A}
\def\F{\mathcal F}
\def\R{\mathcal R}
\def\I{\mathcal I}
\def\J{\mathcal J}

\newcommand{\p}{\mathfrak p}
\newcommand{\q}{\mathfrak q}
\newcommand{\m}{\mathfrak m}

\theoremstyle{plain}
\newtheorem{theorem}[equation]{Theorem}

\newtheorem{proposition}[equation]{Proposition}
\newtheorem{lemma}[equation]{Lemma}

\theoremstyle{definition}
\newtheorem{notation}[equation]{Notation}

\newtheorem{question}[equation]{Question}

\ncom{\bib}{\bibitem}

\ncom{\limns}{\underset{\underset{s}{\longrightarrow}}{\lim}}
\ncom{\limnr}{\underset{\underset{r}{\longrightarrow}}{\lim}}

\ncom{\Tprime}{T^{\prime}}
\ncom{\mprime}{\m^{\prime}}

\def\ff{{\bf f}}

\begin{document}
 \title[ ] { Symbolic Blowup algebras of  monomial curves in $\A^3$ defined by arithmetic sequence }
 \author{Clare D'Cruz}
 \address{Chennai Mathematical Institute, Plot H1 SIPCOT IT Park, Siruseri, 
Kelambakkam 603103, Tamil Nadu, 
India
} 
\email{clare@cmi.ac.in} 
\keywords{Symbolic  Rees algebra, Cohen-Macaulay, Gorenstein}
 \thanks{The author was partially funded by a grant from Infosys Foundation}
 \subjclass[2010]{Primary: 13A30, 1305, 13H15, 13P10} 

\begin{abstract}
In this paper,  we consider monomial curves in  $\A_k^3$ parameterized by $t \rar (t^{2q +1}, t^{2q +1 + m}, t^{2q +1 +2 m})$ where $gcd( 2q+1,m)=1$. The symbolic
blowup algebras   of these monomial curves is Gorenstein (\ \cite{goto-nis-shim}, \cite{goto-nis-shim-2}). From the results in  \cite{schenzel},  it follows that the defining ideal  of 
$R_{s}(\p)$ is given by the Pfaffians of a skew-symmetric matrix. In this paper we give a proof which mainly involves the use of 
Gr$\ddot{o}$ebner basis. We  also able  describe all the symbolic powers $\p^{(n)}$ for all $n \geq 1$. 
\end{abstract}

\maketitle

\section{introduction}
Let $\kk$ be a field and let $\A^n_{\kk}$ (or $\A^n$) be the affine  $n$-space over $\kk$. 
One
well known question is: Is every affine irreducible curve $Y$ in $\A^n_{\kk}$ a set theoretic complete intersection of $(n - 1)$
hypersurfaces?
 In other words, if  $T = k[x_1, \ldots , x_n]$ and  $I(Y ) =\{f \in  T | f (a) = 0 \mbox{ for all } a \in Y \}$, 
 then does there exist $n - 1$ elements $f_1, \ldots , f_{n-1} \in T$ such that 
 $I(Y) = \sqrt{(f_1, \ldots , f_{n-1}})$?
The answer to this question is quite difficult and depends  upon the characteristic of the field $\kk$.
The oldest known result in this direction is a result of L.~Kronecker where he showed that $I(Y)$ can be generated set theoretically by $(n+1)$-equations \cite{kronecker}. 

Later, several researchers showed that there exist  interesting examples  of algebraic subsets $Y \subseteq 
\A^n$  which can be set theoretically defined by $(n-1)$-equations. One remarkable result  which appeared in  
1978 was by  R.~Cowsik and M.~Nori.  They showed that if $\kk$ is  a perfect field such that $\charac(\kk) = p$ 
and if $I$ is a radical ideal  of codimension one, then   $I$ is a set theoretic complete intersection \cite[Theorem 
1]{cowsik-nori}. Later in 1979, for $\charac(\kk)=0$, H.~Bresinsky  showed that  all monomial curves in $\A^3_{\kk}$ are set theoretic complete intersection \cite{bresinsky}. In \cite{bresinsky2} the author extended his ideas to monomial curves in $\A^4$. 
In \cite{scolan} the author extends the ideas in \cite{schenzel2} to study symbolic powers of monomial curves in $\A^4$. 
In 1990, D.~Patil gave   more general class of  monomial curves in $\A^n$ which are set theoretic complete intersection \cite[Theorem~1.1]{patil}. 
 
In $1981$, R.~Cowsik gave a new direction to this problem. He showed that if
$(R, \m)$ a regular local ring and $\p$ is a prime ideal such that $\dim(R/ \p)=1$, then  Noetherianness of  the symbolic Rees Algebra
$R_{s}(\p) := \oplus_{n \geq 0} \p^{(n)}$ implies that  $\p$ is a set theoretic complete intersection 
\cite{cowsik}.  However, the converse need not be true \cite{goto-nis-wat}. Motivated by 
Cowsik's result, 
 in 1987,  Huneke gave necessary and sufficient conditions for 
$\R_{s}(\p)$ to be Noetherian when $\dim~ R=3$ \cite{huneke}. Huneke's result  was 
generalised in $1991$ for $\dim~ R \geq 3$ by M.~Morales  \cite{morales}. All these 
results paved a new  way to  study the famous problem of set theoretic complete 
intersection. In the last twenty-five years several researchers have worked on the  symbolic Rees Algebra $\R_{s}(\p)$. However, there are still many unanswered questions related both the problems, i.e., set theoretic complete intersection and Noetherianness of symbolic Rees algebra.  A few problems will be listed at the end of this paper.  A good survey article with some open questions on set theoretic complete intersection is by G.~Lyubeznik \cite{lyubeznik}.

In this paper, we consider the monomial curve in $\A^3$ parameterized by $(t^{n_1}, t^{n_2}, t^{n_3})$ where  $n_i = 2q + 1 + (i-1)m$ and $gcd(n_1, m) = 1$.  Throughout this paper   $R= k[[x_1, x_2, x_3]]$,   $S = k[[t]]$  and  $\p= \ker(\phi)$, where   $\phi$ is the homomorphism  defined by $\phi(x_i) = t^{n_i}$ for $1 \leq i \leq 3$. We say that  $\p$ is the defining ideal of the  monomial curve parameterized by 
$(t^{n_1}, t^{n_2}, t^{n_3})$. It is well known that these curves are a set-theoretic complete intersection (\cite {bresinsky}, \cite{valla}, \cite{patil}).
The symbolic powers  of curves  are also of interest  (\cite{eliahou}, \cite{huneke0}, \cite{schenzel}, \cite{schenzel2}, \cite{vasconcelos2}). 
In this paper we describe all the symbolic powers  $\p^{(n)}$ (Theorem~\ref{symbolic power}).

 The symbolic Rees algebra of these monomial curves have been studied by several authors in the past. For example see  \cite{valla},  \cite{herzog-ulrich}, \cite{goto-nis-shim},  \cite{schenzel} and 
 \cite{goto-nis-shim-2}. 
The Cohen-Macaulay and Gorenstein property of these monomial curves has also been studied
 \cite{goto-nis-shim}  and \cite{goto-nis-shim-2}. In \cite{schenzel}, the authors  are able to explicitly describe when $R_{s}(\p)$ is Cohen-Macaulay and  when it is Gorenstein. 

We give a different and simple proof for  the Cohen-Macaulayness and Gorensteinness of the blowup algebras $R_s(\p)$ and $G_s(\p):= \oplus_{n \geq 0} \p^{(n)} / \p^{(n+1)}$. One crucial point here is that  though  the ideals we are interested in  are binomial ideals in $\kk[[ x_1, x_2, x_3]]$, we  can to get our main result, by  dealing  with the corresponding  monomial ideals in  the polynomial ring $\kk[x_2, x_3]$. These ideals are defined in Section~\ref{prelim}. Powers and quotients  of these  monomial ideals  are much easier to compute and do not depend on any sophisticated results.   Hence,  any student with a good knowledge of basic Commutative Algebra can follow the proof. Our proofs may give a different approach and help to answer some questions related to symbolic powers. All basic results  used in this paper can be found in \cite{matsumura}. 

We now describe the organisation of  this paper. In Section~\ref{prelim}, we state some basic results. We also state a result which gives a way to compute lengths of modules in the polynomial ring $k[x_2, x_3]$. In Section~\ref{cm-filtration}, we  prove the Cohen-Macaulayness of the filtration $\F = \{ I_n\}_{n \geq 0}$ where the ideals $I_n \subseteq \kk[x_2, x_3]$ are defined in Section~\ref{prelim}. 
In Section~\ref{inductive}, we compute length of $I_{n-1}/ (I_n:x_3^q)$ which is useful in the computations in the  next section. 
In Section~\ref{crucial},  we compute lengths various modules over $\kk[[ x_1, x_2, x_3]]$ which are needed to prove the Cohen-Macaulay and Gorenstein property of the blowup algebras $R_s(\p)$ and $G_s(\p)$.  In Section~\ref{the main section}, we prove our main results. In Section~\ref{questions} we suggest list a few questions which might be of interest to the reader.

\section{Preliminaries}
\label{prelim}
It is well known that  the generators for $\p$  are the $2 \times 2$ minors of the matrix 
$
{\displaystyle
\begin{pmatrix}
   x_1   &  x_2  &  x_3^q\\
     x_2   &  x_3 &  x_1^{m+q}
\end{pmatrix}
}$ \cite{herzog}. 
In particular, if  $ g_1 = x_1^{m+q}x_2 - x_3^{q+1}$, 
  $g_2 =  x_1^{m+q+1} - x_2 x_3^q$ and 
 $g_3 =   x_2^2 - x_1 x_3$,  then 
  $\p = (g_1, g_2, g_3)$.
  
  The following result is well known (\cite[Corollary~4.4]{goto-nis-shim-2}). We state it for the sake of completion. 
\begin{lemma}
\label{noetherian}
  $R_{s}(\p)$ is a Noetherian ring. 
\end{lemma}  
\begin{proof}  
Let $f_1 := g_3 = x_2^2-x_2x_3 \in \p$ and 
$f_2 = - x_1^{2(m+q)+1}- x_1^{m+q-1} x_2^3 x_3^{q-1}+3 x_1^{m+q} x_2 x_3^{q}-x_3^{2q+1}$. Then 
\beqn
x_3 \cdot f_2 = - g_1^2 + x_1^{m+q-1} g_2g_3 \in \p^2 \subseteq \p^{(2)}.
 \eeqn
As $x_3^n$ is nonzerodivisor on $R/ \p^{(2)}$ for all $n$, $f_2 \in \p^{(2)}$.
Moreover, 
\beqn
\ell \left( \f{R}{(x_1, f_1, f_2)}\right)
= \ell \left( \f{R}{(x_1, x_2^2, x_3^{2q+1})}\right)
= 2( 2q+1) 
= 2 \cdot  e(x_1; R/ \p). 
\eeqn
By Huneke's criterion \cite[Theorem~3.1]{huneke},   $R_{s}(\p)$ is Noetherian.
\end{proof}

 Let $T=\kk[x_1, x_2, x_3]$.  The following lemma gives us a way to compute the length of an $R$-module in terms of the  length of the corresponding $T$-module.

  \begin{lemma}
  \label{comparing lengths}
\cite[Lemma~2.8]{clare-shreedevi} 
  Let
  $\m = (x_1, x_2, x_3)T$
   and $M$ a finitely generated  $T$-module such that  $\supp(M) = \{ \m\}$.  Then 
 \beqn
 \ell_R ( M \otimes_T R) = \ell_T(M).
 \eeqn
  \end{lemma}

Let
${\mathcal J}_1 = \{ g_1, g_2,  g_3 \}$
 and 
${\mathcal J}_2 =  \{ f_2 \}.$
We define 
\beq
\label{equation of Jn} 
\J_1 T := (g_1, g_2, g_3)T, \hspace{.2in}
\J_2 T := (f_2) T, \hspace{.2in}
     {\I}_n T
:= \sum_{a_1 + 2 a_2  = n} 
      (\J_{1}^{a_1}T)(  \J_{2}^{a_{2}} T).
\eeq
As $R$ is a flat $T$-module,
  $ {\mathcal I}_nR = {\mathcal I}_n T \otimes_T  R$.

Let $T^{\prime} = k[x_2, x_3] \cong T/ x_1T$. For $i=1,2$ put  
 \beq
\label{definition of Ji}
J_1 :=  \J_1 T^{\prime} =  (x_2^2, x_2x_3^q, x_3^{q+1}),   \hspace{.2in}
J_2 :=  \J_2 T^{\prime}=  (x_3^{2q+1}), \hspace{.2in}
I_n :=     {\mathcal I}_n T^{\prime}
=  \sum_{a_1 + 2 a_2 = n} J_{1}^{a_1} J_{2}^{a_{2} }.
\eeq

\begin{proposition}
\label{description of In}
Let $n \geq 1$. Then
\been
\item
\label{description  of In one}
$    {\mathcal I}_n  R\subseteq   \p^{(n)}$.

\item
\label{description of In two}
$  ({\mathcal I}_n  + (x_1))T$  is an    homogeneous ideal.  

\item
\label{description of In three}
 $({\mathcal I}_n  + (x_1))T$ is an $\m$-primary ideal.
\eeen
\end{proposition}
\begin{proof}
(\ref{description of In one})
As   ${\J}_1 \p$ and $J_2 = (f_2)    \subseteq \p^{(2)}$,   for all $a_1,  a_{2} \in \Z_{\geq 0}$, 
\beq
\label{containment of J}
 \J_{1}^{a_1} \J_{2}^{a_{2} }
 \subseteq \p^{a_1} (\p^{(2)})^{a_2} 
 \subseteq \p^{(a_1 + 2a_2)  }.
\eeq
Summing over all $a_1+ 2 a_2 =n$ and applying (\ref{containment of J})  to (\ref{equation of Jn}) we get (\ref{description of In one}).

(\ref{description of In two})  
As  $({\mathcal J}_1 + (x_1))T = (J_1 , ( x_1))T$ and 
  $({\mathcal J}_2 + (x_1))T = (J_2 , x_1)T$  are homogenous ideals and 
 $ ({\mathcal I}_n + (x_1))T =   (\sum_{a_1 + 2 a_2 = n} J_{1}^{a_1} J_{2}^{a_{2} } , x_1)T$ we get 
 (\ref{description  of In two}). 

(\ref{description  of In three}) By \eqref{equation of Jn}, ${\mathcal J}_1^{n}T \subseteq {\mathcal I}_nT$ and  $({\mathcal J}_1^{n} + (x_1))T = ((x_2^2, x_2x_3^q, x_3^{q+1})^{n} ,x_1)T$ which implies that 
$\m T =( \sqrt{{\mathcal J}_1^{n} + (x_1)} )T \subseteq (\sqrt{{\mathcal I}_n + (x_1)})T \subseteq \m T$. 
\end{proof}

We state a result on monomial ideals which follows from \cite[Proposition~1.14]{ene-herzog} and will be consistently used in all the proofs which involve monomial ideals. 
\begin{proposition}
\label{prop-herzog}
Let $I = (u_1, \ldots, u_r)$ and $J = (v)$ be monomial ideals  in a polynomial ring over a field $\kk$. 
Then $I : J =  ( \{u_i/gcd(u_i,v) : i = 1,...,r\})$.
\end{proposition}

\section{The associated graded ring corresponding to the filtration $\F:= \{ I_n\}_{n \geq 0}$.}
\label{cm-filtration}
Throughout this section we will work with the ring $T^{\prime}= \kk[x_2, x_3]$ and the ideals $J_1$, $J_2$ and $I_n$. Our goal is to compute  $\ell (T^{\prime}/ (I_{n} + (x_2)^2))$ and
$\ell (T^{\prime}/ ( I_{n} + (x_2^{2}, x_3^{2q+1} ) )) $. To attain our goal, we 
 show that $x_2^2$, $x_3^{2q+1}$ is  a regular sequence in $G(\F)$, where  $G(\F) := \oplus_{n \geq 0} I_n / I_{n+1}$ is the associated graded ring  corresponding to the filtration  ${\F}:= \{ I_n\}_{n \geq 0}$.

\begin{theorem}
\label{cohen macaulayness of G}
$G(\F)$ is Cohen-Macaulay.
\end{theorem}
\begin{proof} We first show that $(I_n : x_2^2)  = I_{n-1}$ for all $n \geq 2$. 
Clearly $x_2^2I_{n-1} \subseteq J_1 I_{n-1} \subseteq I_n$. Write 
$J_1^{a_1} 
= \sum_{i=0}^{a_1-1} x_2^{2(a_1-i)} x_3^{qi} (x_2, x_3)^i  + x_3^{qa_1} (x_2, x_3)^{a_1}$. Then 
\beq
\label{colon with x_2^2} \nno
&&         (I_n : x_2^2) \\ \nno
&=& \left(  \left( \sum_{a_1 + 2a_2 = n} J_1^{a_1} J_2^{a_2} \right) : x_2^2 \right)\\ \nno
&=&  \sum_{a_1 + 2a_2 = n}( J_1^{a_1} J_2^{a_2} : x_2^2)    \\ \nno
&=& \sum_{a_1 + 2a_2 =n}
        \sum_{i=0}^{a_1-1} (x_2^{2(a_1-i)} x_3^{q i + (2q+1)a_2} (x_2, x_3)^i   : x_2^2)
 +                            (    (x_3^{q a_1+(2q+1)a_2} (x_2, x_3)^{a_1} ): x_2^2) \\ \nno
 &=& \sum_{a_1 + 2a_2 =n}
        \sum_{i=0}^{a_1-1} (x_2^{2(a_1-i-1)} x_3^{q i + (2q+1)a_2} (x_2, x_3)^i )
        +  (x_3^{q a_1+(2q+1)a_2} (x_2, x_3)^{a_1-1} )  \\
         &\subseteq& I_{n-1}. 
\eeq
Let $\olin{\hphantom{xx}}$ denote the image in $R/ (x_2^2)$. Then
\beqn
\f{G(\F) }{({(x_2^2)}^{\star} ) }
&\cong& \bigoplus_{n \geq 0} \f{I_n}
                                                 { I_{n+1} +{x_2^2} I_{n-1}} = G( \olin{\F}).
                                                 \eeqn
 To  show that $\olin {x_3^{2q+1}}$  is a regular element in $G(\olin{\F})$, we  need to verify that  
 \beq
 \label{x3 is regular}    ((I_{n+2}  + {x_2^{2} I_{n+1}) : (x_3^{2q+1})) } 
 = I_{n} + x_2^2  I_{n-1}.
 \eeq  
 Let $m \geq 1$. Then
 \beq
\label{colon of j1 with x3} \nno
&& (J_1^{m+2} : x_3^{2q+1}) \\ \nno
&=&               \sum_{i=0}^{2} (  x_2^{2(m+2-i)}  x_3^{qi} (x_2, x_3)^{i} : x_3^{2q+1})
+                    \sum_{i=3}^{m+2} (  x_2^{2(m+2-i)}  x_3^{qi} (x_2, x_3)^{i} : x_3^{2q+1})  \\ \nno
&=&                   x_2^{2m}(x_2, x_3) 
+                    \sum_{i=3}^{m+2} (  x_2^{2(m+2-i)}  x_3^{q(i-3)} (x_2, x_3)^{i-3})
                        (x_2^3x_3^{q-1}, x_2^2x_3^q,  x_2x_3^{q+1}, x_3^{q+2}) \\ \nno
&\subseteq&   ((x_2^{2m)})
 +                   \sum_{i=3}^{m+2} (  x_2^{2(m+2-i)}  x_3^{q(i-3) +1} (x_2, x_3)^{i-3}) (x_2^2, x_2x_3^q, x_3^{q+1})\\
 &=& J_1^{m}.
\eeq
Hence 
\beq
\label{x3 is regular computation}\nno
&&                    { \displaystyle ((I_{n+2}  + {x_2^2} I_{n+1}) : (x_3^{2q+1})) } \\ \nno
 &=&               \sum_{a_1+ 2a_2 = n+2} (J_1^{a_1} J_2^{a_2} : x_3^{2q+1})
 +                    \sum_{a_1+ 2a_2 = n+1}( {x_2^{2} J_1^{a_1} J_2^{a_2}: (x_3^{2q+1}))}   \\ \nno
 &\subseteq&  \sum_{a_1+ 2a_2 = n+2; a_2 \not =0} (J_1^{a_1} J_2^{a_2-1})
  +                       (J_1^{n+2} : x_3^{2q+1}) \\ \nno
&& +                    \sum_{a_1+ 2a_2 = n+1; a_2 \not = 0}(x_2^2  J_1^{a_1} J_2^{a_2-1} )
 +                       (x_2^2J_1^{n+1} : x_3^{2q+1})   \\ \nno
 &\subseteq & \sum_{a_1+ 2a_2 = n+2; a_2 \not =0} (J_1^{a_1} J_2^{a_2-1})
  + J_1^n +    \sum_{a_1+ 2a_2 = n+1; a_2 \not = 0}(  x_2^2J_1^{a_1} J_2^{a_2-1} )
+ x_2^2J_1^{n-1} 
    \hspace{1.2in} \mbox{[by \eqref{colon of j1 with x3}]} \\ 
    &\subseteq &  I_{n} + x_2^2  I_{n-1}.
\eeq
The other inclusion is easy to verify. 
\end{proof}

We are now ready to prove the main result of this section.
\begin{proposition}
\label{computing full length}
For all $n \geq 1$,

\beqn
    \ell   \left( \f{\Tprime}{(I_{n} + (x_2^2 )  )\Tprime} \right)
&=&  \ell \left( \f{\Tprime}{(I_{n} )\Tprime}  \right) 
-  \ell \left( \f{\Tprime}{(I_{n-1} )\Tprime}  \right) ,\\
    \ell   \left( \f{\Tprime}{(I_n + (x_2^2, x_3^{2q+1} )  )\Tprime} \right)
&=&  \ell \left( \f{\Tprime}{(I_{n} )\Tprime}  \right)
 -  \ell \left( \f{\Tprime}{(I_{n-1} )\Tprime}  \right)
-  \ell \left( \f{\Tprime}{(I_{n-2} )\Tprime}  \right)
+   \ell \left( \f{\Tprime}{(I_{n-3} )\Tprime}  \right).
\eeqn
\end{proposition}
\begin{proof} The proof follows from  \cite[Proposition~2.4]{clare-shreedevi} and Theorem~\ref{cohen macaulayness of G}.
\end{proof}

\section{The inductive step}
In this section we describe the generators of $I_{n-1}$ modulo $(I_n : x_3^q)$. This will be used in our computations in the next section. 
\label{inductive}
\begin{lemma}
\label{ideal containment}
For all $n \geq 1$,
\been
\item
\label{ideal containment 1}
$ (I_n : x^q_3)  \subseteq I_{n-1}$.
\item
\label{ideal containment 2}
${\displaystyle
I_{n-1} = 
 \begin{cases}
 {\displaystyle \sum_{a_2=0}^{ \frac{n-2}{2}}    
   x_2^{2(n-1-2a_2) -1}   x_3^{(2q+1) a_2}  (x_2, x_3^q)   
+  (I_{n} : x_3^q )  }& \mbox{ if }2 \not |  (n-1 )\\
   \left( x_3^{(2q+1) \left( \frac{n-1}{2}\right)} \right)
+     {\displaystyle \sum_{a_2=0}^{ \frac{n-3}{2} }    
        x_2^{2(n-1-2a_2) -1}  x_3^{(2q+1)a_2}   (x_2, x_3^q)    
 +  (I_{n} : x_3^q )  } 
 &   \mbox{ if }2 |  (n-1 )
\end{cases}
}.
$
\eeen
\end{lemma}
\begin{proof}
\eqref{ideal containment 1} One can verify that 
\beqn
                        (I_n : x^q_3 )
&=&               \sum_{a_1+2a2=n; a_2 \not = 0} (J^{a_1}_1 J^{a_2}_2 : x^q_3) + (J_1^n:x_3^q)\\
&=&                \sum_{a_1+2a2=n;a2 \not =0}(x^{q+1}_3 )J^{ a_1}_1 J^{a_2-1}_2 
+                      (   x^{2n}_2) 
+                     \sum_{i=1}^n (x^{2(n-i)}_2 x^{q(i-1)}_3 (x_2, x_3^q)^i  )\\
&\subseteq &   I_{n-1}.
\eeqn

\noindent
\eqref{ideal containment 2}
As 
$                       I_1 = J_1 
=                x_2(x_2, x_3^q) + (x_3^{q+1}) \subseteq x_2(x_2, x_3^q)  +( I_2: x_3^q),
$
\eqref{ideal containment 2} is true for $n=1$. Let $n>1$. For all $a_1 \geq 1$, 
\beq
\label{in between step} \nno
                      && J_1^{a_1}  \\ \nno
&=&                 J_1 J_1^{a_1-1}\\ \nno
& \subseteq& (x_2(x_2, x_3^q), x_3^{q+1} )  \left(  x_2^{2a_1-3}  (x_2, x_3^q) + (I_{a_1} : x_3^q )\right)
                         \hspace{1.8in} \mbox{[by induction hypothesis]}     \\   \nno 
&=&            x_2^{2a_1-1}  (x_2, x_3^q)     +     (x_2^{2a_1-2} x_3^{2q} )+   x_2(x_2, x_3^q)       (I_{a_1} : x_3^q )      
 +       x_3^{q+1} x_2^{2a_1-3}  (x_2, x_3^q) +    x_3^{q+1} (I_{a_1} : x_3^q )    \\    
   &=&              x_2^{2a_1 -1}     (x_2, x_3^q)     + (I_{a_1 + 1} : x_3^q )   
                      \eeq
                     as
                     \beqn
                      (x_2^{2a_1-2} x_3^{2q} )x_3^q 
                      &=& x_2^{2(a_1-1)} x_3^{2q+1} x_3^{q-1} \subseteq J_1^{a_1-1} J_2^{a_2} \subseteq I_{a_1 -1 +2} = I_{a_1 + 1}\\
  \left( x_3^{q+1} x_2^{2a_1-3}  (x_2, x_3^q)   \right) x_3^q 
  &=    &  x_2^{2(a_1-2)} \cdot x_2(x_2, x_3^q) \cdot x_3^{2q+1}
  \subseteq   J_1^{a_1-1} J_2^{a_2} \subseteq I_{a_1 -1 +2} = I_{a_1 + 1}\\
     x_3^{q+1} (I_{a_1} : x_3^q )&  \subseteq   &   (I_{a_1+1} : x_3^q )        
                     \eeqn
Hence
\beqn
&&I_n \\
&=& 
\begin{cases}
{\displaystyle \sum_{a_2=0}^{ \frac{n-2}{2}}  J_1^{n-1-2a_2 } J_2^{a_2}} & 
    \mbox{ if }2 \not |  (n-1 )\\ 
J_2^{(n-1)/2}  
+  {\displaystyle \sum_{a_2=0}^{ \frac{n-3}{2} } J_1^{n-1-2a_2 } J_2^{a_2}} 
&  \mbox{ if }2 |  (n-1 )
\end{cases}\\
&\subseteq &
 \begin{cases}
 {\displaystyle \sum_{a_2=0}^{ \frac{n-2}{2}}    
   x_2^{2(n-1-2a_2) -1}   x_3^{(2q+1) a_2}  (x_2, x_3^q)   
+   x_3^{(2q+1) a_2}(I_{n-2a_2} : x_3^q )  }& \mbox{ if }2 \not |  (n-1 )\\
   x_3^{(2q+1) \left( \frac{n-1}{2}\right)} 
+     {\displaystyle \sum_{a_2=0}^{ \frac{n-3}{2} }    
        x_2^{2(n-1-2a_2) -1}  x_3^{(2q+1)a_2}   (x_2, x_3^q)    
 + x_3^{(2q+1)a_2} (I_{n-2a_2} : x_3^q )  } 
 &   \mbox{ if }2 |  (n-1 )
  \end{cases} \mbox{[by \eqref{in between step}]}\\
  &\subseteq &
 \begin{cases}
 {\displaystyle \sum_{a_2=0}^{ \frac{n-2}{2}}    
   x_2^{2(n-1-2a_2) -1}   x_3^{(2q+1) a_2}  (x_2, x_3^q)   
+  (I_{n} : x_3^q )  }& \mbox{ if }2 \not |  (n-1 )\\
   x_3^{(2q+1) \left( \frac{n-1}{2}\right)} 
+     {\displaystyle \sum_{a_2=0}^{ \frac{n-3}{2} }    
        x_2^{2(n-1-2a_2) -1}  x_3^{(2q+1)a_2}   (x_2, x_3^q)    
 +  (I_{n} : x_3^q )  } 
 &   \mbox{ if }2 |  (n-1 ).
  \end{cases}
\eeqn
This implies that $I_n \subseteq RHS$. The other inclusion follows from  \eqref{ideal containment 1} and checking element-wise. 
\end{proof}

\begin{proposition}
\label{length of last term}
For all $n \geq 1$,
\beqn
\dim_{\kk} \left( \f{I_{n-1} }
                  { (I_n : x_3^q)} \right) = n. 
\eeqn
\end{proposition}
\begin{proof}
Put $\m^{\prime} = (x_2, x_3)$. Then $x_3^q \m^{\prime} I_{n-1} \subseteq J_1 I_{n-1} \subseteq I_n$.  Hence 
by  Lemma~\ref{ideal containment}\eqref{ideal containment 2} we get 
\beqn
\dim_{\kk} \left( \f{I_{n-1} }
                  { (I_n : x_3^q)} \right)
&\leq & \begin{cases}
2(n/2)  & \mbox{ if } 2 \not  |n-1 \\
1 + [2(n-1)/2]  & \mbox{ if } 2 | n-1
\end{cases}
= n. 
\eeqn
To complete the proof we need to show that we have $n$ linearly independent elements. As all the elements in 
$I_{n-1} / (I_n : x_d)$ has the degree of $x_2$   different, they form a linearly independent set. 
\end{proof}

\section{Cohen-Macaulayness of $R/  (\p^{(n)}    + (\ff_k))$ }
\label{crucial}
Let $(\ff_1) := (f_1)$ and $ (\ff_2) := (f_1, f_2)$, where $f_1$ and $f_2$ are defined in  Lemma~\ref{noetherian}. 
The main step  in proving the Cohen-Macaulayness of $\R_{s}(\p)$ is to show that or all $n \geq 1$ and $ k =1,2$,  the rings  $R/  (\p^{(n)}    + (\ff_k) )$  are Cohen-Macaulay \cite{goto}. 
This was proved by Goto for $q=1$ and $n \leq {3}$ (\cite[Proposition~7.6]{goto}). We prove it for all $q \geq 1$ and for all $n$. 
As a consequence, we   prove that 
 $\p^{(n)} = {\mathcal I}_nR$ and
$(\p^{(n)} \Tprime) = I_n\Tprime$ for all   $n \geq 1$.  

We first compute the length of the  $T^{\prime}/ I_n$.  Next, we  prove an interesting result which gives the  the equality of the  lengths of the various modules
(Theorem~\ref{main theorem 1}, Theorem~\ref{bound on length}). 

\begin{proposition}
\label{main theorem}
For all $n \geq 1$, 
\beqn
      \ell \left( \f{\Tprime}{I_n} \right)
=     (2q+1){n+1 \choose 2}.
\eeqn
\end{proposition}
\begin{proof} We prove induction on $n$. If $n =1$, then 
\beqn
\ell \left( \f{\Tprime}{I_1} \right)
= \ell \left( \f{k[x_2,  x_3]} {(x_2^2, x_2 x_3^q, x_3^{q+1}) }\right)
 = 2q+1.
 \eeqn
Now let $n > 1$.
From the  exact sequence
\beqn
        0 
\lrar \f{\Tprime}{(I_n: x_3^q)} 
\stackrel{.x_3^q}{\lrar} \f{\Tprime}{I_n}
\lrar \f{\Tprime}{I_n + (x_3^q)}
\lrar 0  
\eeqn
we get
\beqn
&&        \ell \left( \f{\Tprime}{I_n}\right)\\
&=&      \ell \left( \f{\Tprime}{I_n + (x_3^q)}\right)
+            \ell \left( \f{\Tprime}{(I_n : x_3^q)}\right)\\
&=&       \ell \left( \f{\Tprime}{ (x^{2n}, x_3^q)}\right)
+            \ell \left( \f{\Tprime}{I_{n-1}}\right)
+            \ell \left( \f{I_{n-1}}{(I_n : x_3^q)}\right) 
              \hspace{2.17in} \mbox{[Lemma~\ref{ideal containment}(\ref{ideal containment 1})]}\\
&=&        2qn
+             (2q+1)  {n \choose 2}
+            n 
             \hspace{1.4in} 
             \mbox{[by induction hypothesis and Proposition~\ref{length of last term}]}\\
&= &      (2q+1) {n + 1 \choose 2}.
\eeqn
\end{proof}

\begin{theorem}
\label{main theorem 1}
For all $n \geq 1$, 
\beqn
    e   \left(x_1; \f{R}{\p^{(n)}}\right)
= \ell \left(  \f{R}{\p^{(n)} + (x_1)}\right)
&=&  \ell_R \left(\f{R}
                        { ({\mathcal I}_{n}, x_1) R}\right)
=    \ell_{\Tprime} \left(  \f{\Tprime} 
                                            { {\mathcal I}_{n}\Tprime }  
                                            \right)
= \ell \left( \f{\Tprime}{I_n} \right)
= (2q+1){n+1 \choose 2}.
\eeqn
\end{theorem}
\begin{proof} 
From  Proposition~\ref{description of In}(\ref{description  of In one}),   ${\mathcal I}_n  R\subseteq   \p^{(n)}$.  
Hence,
\beq
\label{equality of all terms 1}\nonumber
 e\left(x_1;\f{R}
                         {\p^{(n)}}\right)
&=&    \ell_R \left(\f{R}
                                  {\p^{(n)}+(x_1)}\right) 
                                    \hspace{1.8in} [\mbox   {as $R/ \p^{(n)}$  is Cohen-Macaulay}]\\
 &\leq&         \ell_R \left(\f{R}
                                  { ({\mathcal I}_{n}, x_1) R}\right).            
                                  \eeq
  By 
 Proposition~\ref{description of In}(\ref{description  of In three}), for any prime $\q \not = \m $,  $(({\mathcal I}_n , x_1)T)_{\q} = T$. Hence 
${\displaystyle \supp_T \left(  \f{T  }{ ( {\mathcal I}_{n}, x_1)T} \right) =\{\m\}} $
 and  we get  
  \beq
\label{equality of all terms 2}\nonumber
            \ell_R \left(\f{R}
                                  { ({\mathcal I}_{n}, x_1) R}\right)    \nonumber                   
&=&       \ell_{\Tprime} \left(  \f{\Tprime} 
                                            { {\mathcal I}_{n} \Tprime }  
                                            \right)              \hspace{3.05in} \mbox{[Lemma~\ref{comparing lengths}]}     \\    \nonumber
 &=&              \ell_{\Tprime} 
              \left(\f{\Tprime}
                       {I_n}\right)      \hspace{3.55in} 
               \mbox{  [\eqref{definition of Ji}]  }  \\ \nonumber    
& = &    (2q+1) \binom{n+1}{2}   
  \hspace{2.4in} 
               \mbox{  [Proposition~\ref{main theorem}]} \\     \nonumber    
 &=& e \left(x_1;\f{R}{\p}\right)
              \ell_{R_{\p}}
              \left(\f{R_{\p}}{\p^{n}R_{\p}}\right) \\ \nonumber
&=& e \left(x_1;\f{R}
                     {\p}\right)
             \ell_{R_{\p}}
             \left(\f{R_{\p}}
                       {\p^{(n)}R_{\p}}\right)     
                        \hspace{1.6in} [\mbox{since } \p^{(n)}R_{\p}=\p^{n}R_{\p}]\\          
&=&    e\left(x_1;\f{R}
                         {\p^{(n)}}\right)
                \hspace{2.3in} \mbox{[by \cite[Theorem 14.7]{matsumura}]}.    
\eeq
Thus equality holds in (\ref{equality of all terms 1})  and  (\ref{equality of all terms 2}) which proves the theorem.
\end{proof}

\begin{notation}
Let 
$\ff_1= f_1$, $\ff_2 = f_1, f_2$, 
$({\bf x_2}) =(x_2^2)$ and $({\bf x_3}) = (x_2^2, x_3^{2q+1})$ . 
\end{notation}

\begin{theorem}
\label{bound on length}
 Let $k=1,2$. Then for all $n \geq 1$, 
\beqn
   e   \left(x_1; \f{R}{\p^{(n)}   + (\ff_k)}\right)
= \ell_R \left(  \f{R}{\p^{(n)} + (x_1, \ff_k)}\right)
=    \ell_{\Tprime} \left(  \f{\Tprime} 
                                            { ({\mathcal I}_{n}  + \ff_k)\Tprime }  
                                            \right)
= \ell \left( \f{\Tprime}{I_n + ({\bf x_{k+1} )  } }\right).
\eeqn
In particular, 
$
 {\displaystyle \f{R}
{\p^{(n)}+(\ff_k )}}
$
is Cohen-Macaulay. 
\end{theorem}
\begin{proof}
From  Proposition~\ref{description of In}(\ref{description  of In one})  $({\mathcal I}_n  ,x_1,\ff_k ) R\subseteq  
( \p^{(n)} ,x_1,\ff_k ) R$.  
Hence 
 \beq
\label{multiplicity inequality 1}\nonumber
                 e \left(x_1; \f{R}
                            {\p^{(n)} + (\ff_k)} \right)  
&\leq &  \ell_R \left( \f{R}
                               {\p^{(n)}+(x_1,\ff_k )}\right) \hspace{1.6in} \mbox{\cite[Theorem~14.10]{matsumura}}\\
&\leq&   \ell_R \left(  \f{R} 
                                 { ({\mathcal I}_{n}, x_1, \ff_k)R }\right) .
\eeq
Since $({\mathcal I}_n, x_1) T \subseteq ({\mathcal I}_n, \ff_k, x_1) T$  by Proposition~\ref{description of In}(\ref{description  of In three}), for any prime $\q \not = \m $,  $(({\mathcal I}_n , \ff_k, x_1)T)_{\q} = T$. This implies that 
${\displaystyle \supp_T \left(  \f{T} { ({\mathcal I}_{n},  \ff_k, x_1)T } \right) =\{\m\} }$.   Hence  we get 

\beq   
\label{multiplicity inequality 2}
\nonumber     
&&   \ell_R \left(  \f{R} 
                                 { ({\mathcal I}_{n},x_1, \ff_k)R }\right)\\\nonumber
&=&        \ell_T \left(  \f{T} 
                                  {( {\mathcal I}_{n}, x_1, \ff_k) T}  \right)    
               \hspace{3.6in} \mbox{[Lemma~\ref{comparing lengths}]}    
              \\ \nonumber
&=&  \ell_{\Tprime} \left(  \f{\Tprime} 
                                              {  { I}_{n} + (  {\bf x_{k+1}} )}  \right) 
                              \\ \nonumber
&=&    
\begin{cases}
            \ell \left( \f{\Tprime}{(I_{n} )\Tprime}  \right) 
-           \ell \left( \f{\Tprime}{(I_{n-1} )\Tprime}  \right) ,& \mbox{ if } k=1\\
            \ell \left( \f{\Tprime}{(I_{n} )\Tprime}  \right)
 -          \ell \left( \f{\Tprime}{(I_{n-1} )\Tprime}  \right)
-           \ell \left( \f{\Tprime}{(I_{n-2} )\Tprime}  \right)
+          \ell \left( \f{\Tprime}{(I_{n-3} )\Tprime}  \right)  & \mbox{ if } k=2
              \end{cases}
               \hspace{0.25in} \mbox{  [Proposition~\ref{computing full length}]}   \\ \nonumber    
&=&
\begin{cases}
  \ell \left( \f{R_{\p}}{\p^{n} R_{\p}}  \right) -   \ell \left( \f{R_{\p}}{\p^{n-1} R_{\p}} \right)    
  & \mbox{ if } k=1\\ 
   \ell \left( \f{R_{\p}}{\p^{n} R_{\p}}  \right) -   \ell \left( \f{R_{\p}}{\p^{n-1} R_{\p}} \right)  
-   \ell \left( \f{R_{\p}}{\p^{n-2} R_{\p}} \right)  +    \ell \left( \f{R_{\p}}{\p^{n-3} R_{\p}} \right) 
  & \mbox{ if } k=2     
  \end{cases}
    \hspace{0.9in}  \mbox{[Theorem \ref{main theorem 1}]}    \\
&=& e \left(x_1; \f{R}
                            {\p^{(n)} + (\ff_k)} \right)   
                            \hspace*{3.4in} \mbox{\cite[Corollary~2.6]{clare-shreedevi}}.                                  
  \eeq 
  Hence equality holds in \eqref{multiplicity inequality 1} and \eqref{multiplicity inequality 2} which proves the theorem.
\end{proof}

We end this section by explicitly describing the generators of $\p^{(n)}$ for all $n \geq 1$. 

\begin{theorem}
\label{symbolic power}
For all $n \geq 1$,  $\p^{(n)} = {\mathcal{I}}_n R$; 
 ${\displaystyle  \p^{(n)}
= \sum_{a_1 + 2a_2=n} 
  \p^{a_1} (\p^{(2)})^{a_2}  }$ and 
  $\p^{(n)} T^{\prime}= I_n =  {\mathcal I}_n \Tprime$. 
\end{theorem}
\begin{proof} 
By Theorem~\ref{main theorem 1}  we get 
$\p^{(n) }  + (x_1)= \I_n R + (x_1).  $ 
Hence  $\p^{(n)} =  {\mathcal I}_n R + x_1 (  \p^{(n)}  :(x_1))$. As $x_1$ is a nonzerodivisor on $R/ \p^{(n)}$, 
 $(  \p^{(n)}  :(x_1)) = \p^{(n)}$. By Nakayama's lemma, $\p^{(n)} =  {\mathcal I}_nR$. 
 
  For all $n \geq 3$, applying Proposition~\ref{description of In}(\ref{description  of In one})   we get
  \beqn
  \p^{(n) } 
  = {\mathcal I}_n R
  = \sum_{a_1 + 2a_2=n} 
           {\mathcal J}_1^{a_1}{\mathcal J}_2^{a_2} R           
 \subseteq      \sum_{a_1 + 2a_2=n} 
         \p^{a_1} (\p^{(2)})^{a_2}
         &\subseteq& \p^{(n)}. 
  \eeqn
 Hence equality holds. 
 The last equality follows from  Theorem~\ref{main theorem 1}. 
\end{proof}

\section{Cohen-Macaulayness  and Gorensteinness of symbolic blowup algebras}
\label{the main section}
 In this section we show that both symbolic blowup  algebras $G_{s}(\p)$ and $\R_s(\p)$ are Cohen-Macaulay and Gorenstein. 
From Proposition~3.7 in \cite{goto-nis-shim-2} it follows that $G_{s}(\p)$ is Cohen-Macaulay. We use the results in this paper to prove it. 

\begin{theorem}
\label{cm-ass-gr}
\been
\item
\label{cm-ass-gr-one}
$G_{s}(\p)$ is Cohen-Macaulay.
\item
\label{cm-ass-gr-two}
$G_{s}(\p)$ is Gorenstein.
\eeen
 \end{theorem}
 \begin{proof} (1) 
By \cite[Corollary~3.2]{goto-nis-shim-2}, it is enough to show that  $x_1^{\star}, f_1^{\star},  f_{2}^{\star}$  is a regular sequence in $G_{s}(\p)$. 
For all $n ]geq 1$,   $x_1^t$ is a nonzerodivisor on $R/ \p^{(n)}$ for all $t$, Hence $x_1^{\star}$ is a regular element in $G_s(\p)$ and 
\beqn
G_s \left(\f{\p+x_1R}{x_1R} \right)
\cong \bigoplus_{n \geq 0} \f{  \p^{(n)} + x_1R }{  \p^{(n+1)} + x_1R }
\cong \bigoplus_{n \geq 0} \f{  \p^{(n)} }{  \p^{(n+1)} + x_1 \p^{(n)} }
\cong \f{G_s(\p)}{x_1^{\star} G_s(\p)}.
\eeqn
To show that $f_1^{\star}$ is a regular sequence we need to show that $((\p^{(n)} + x_1R) : f_1 )= \p^{(n)} + x_1R$. 
Now
 \beqn
          \ell \left( \f{R}
                         {((\p^{(n+1)}, x_1 )  : (f_1))}  \right) 
 &=&  \ell \left( \f{R}{(\p^{(n+1)}, x_1 )} \right) 
 -        \ell \left( \f{R}{(\p^{(n+1)}, x_1, f_{1} )} \right)  
          \\
 &=&  \ell \left( \f{T^{\prime}}{ I_{n+1}  }\right)
 -        \ell \left( \f{T^{\prime}}{ I_{n+1} + ( x_2^2 )}\right) 
 \hspace{1.4in} \mbox{[Theorem~\ref{bound on length}]}\\
 &=&  \ell \left( \f{T^{\prime}}{  (I_{n+1}  : x_{2}^{2}   )}\right)\\
 &=&  \ell \left( \f{T^{\prime}}{ I_{n}}\right) 
             \hspace{2.2in} \mbox{[proof of Theorem~\ref{cohen macaulayness of G}]} \\
 &=&  \ell \left( \f{R}
                         {(\p^{(n)}, x_1 )} \right) \hspace{2.3in} \mbox{[Theorem~\ref{bound on length}]}.
 \eeqn
 Similarly, one can show that 
 $((\p^{(n+1)}, x_1, f_1 )  : (f_2)) = (\p^{(n-1)}, x_1 , f_1)$ which will imply that $f_2^{\star}$ is a nonzerodivisor on $G_s(\p) / (x_1^{\star}, f_1^{\star})$.
 This proves (\ref{cm-ass-gr-one}). 
 
 (\ref{cm-ass-gr-two}) As $G(\p R_{\p})$ is a polynomial  ring, it is Gorenstein. Hence the result follows from Theorem~\ref{bound on length} and \cite[Corollary~5.8]{goto}. 
 \end{proof}
 
 \begin{theorem} 
(\cite[Theorem~4.1]{goto-nis-shim-2}, \cite[Theorem~2]{schenzel}
 \label{cm-rees}
\been
\item
 \label{cm-rees-zero}
  $\R_s(\p)= R[\p t, {\mathcal J_2}t^2] =R[\p t, f_2t^2]  $.
\item
 \label{cm-rees-one}
 $\R_s(\p)$ is Cohen-Macaulay.

\item

 \label{cm-rees-two}
$\R_s(\p)$ is Gorenstein.
\eeen
\end{theorem}
\begin{proof}
(\ref{cm-rees-zero}) The proof  follows from  Theorem~\ref{symbolic power}.

(\ref{cm-rees-one})
By Theorem~\ref{bound on length}, ${\displaystyle \f{R}{\p^{(n)}+(\ff_{2})}}$ is Cohen-Macaulay for all $n \geq 1$. Hence,
by \cite[Theorem~6.7]{goto}, $\R_s(\p)$ is Cohen-Macaulay.

(\ref{cm-rees-two})
By \cite[Lemma~6.1]{goto}, the a-invariant of $(G_{s}(\p))$,  $a(G_{s}(\p))= -(2)$. By \cite[Theorem~6.6]{goto}, 
 and Theorem~\ref{cm-ass-gr}, $\R_s(\p)$ is Gorenstein.
\end{proof}

\section{A few questions}
\label{questions}
In this section we state a few related  questions which  are of  interest. 

\begin{question}
In \cite{goto-mor} S.~Goto and M. Morimoto studied the monomial curves $(t^{n^2+2n+2}, t^{n^2+2n+1}, t^{n^2+n+1})$. They showed that for $n=p^r$  ($r 
\geq 1$),  $R_s(\p)$ is Noetherian but not Cohen-Macaulay if $\charac~\kk =p$. In \cite{goto-nis-wat}, the authors  raised the 
following question. If $ \charac~\kk = p^{\prime} \not = p$, is   $R_s(\p)$ Noetherian? In \cite{vasconcelos}, W.~Vasconcleos 
showed that if $\charac~\kk =0$ and $n=2$, then $R_s(\p)$ Noetherian. If $p$ is a prime and $n=p>2$, then is $R_s(\p)$ Noetherian?
\end{question}

\begin{question}
Let $\kk$ be a field and let  $\p$ be the prime ideal defining the monomial curve $(t^a, t^b, c^c)$ in $\A^3$ where $a$, $b$ and $c$ are pairwise coprime. In \cite{cutkosky}, D. Cutkosky gave a geometric meaning to symbolic primes.  
 He found some interesting examples of monomial curves for which  $R_s(\p)$ is finitely generated. Using the criteria in D.~Cutkosky's paper,  H.~Srinivasan showed that 
if $a =6$, then  $R_s(\p)$ is finitely generated  \cite{srinivasan}. In \cite{huneke},   C. Huneke showed that $R_s(\p)$ is finitely generated if $a=4$. 
 Can we find all possible $(a,b,c)$ such that $R_s(\p)$ is finitely generated?

\end{question}

\end{document}